\documentclass[11pt]{amsart}
\usepackage{amsmath,amssymb}
\usepackage[all,cmtip]{xy}
\usepackage{mathtools}

\begin{document}

\theoremstyle{plain}
\newtheorem{thm}{Theorem}[section]
\newtheorem{lem}[thm]{Lemma}
\newtheorem{prop}[thm]{Proposition}
\newtheorem{cor}[thm]{Corollary}
\theoremstyle{definition}
\newtheorem{defn}[thm]{Definition}
\newtheorem{remark}{Remark}
\newtheorem{condition}[thm]{Condition}
\newtheorem{example}[thm]{Example}
\newtheorem{claim}[thm]{Claim}

\newcommand{\oset}{\N \cup \{\omega\}}
\newcommand{\al}{ \alpha }
\newcommand{\monoset}{\Lambda}
\newcommand{\monosetelement}{r}
\newcommand{\monotoneelement}{u}
\newcommand{\monotonelast}{u_{\omega}}
\newcommand{\neib}[2]{N_{#1}(\{#2\})}
\newcommand{\denssetelement}{\Psi}
\newcommand{\realdist}[2]{|{#1}-{#2}|}
\newcommand{\tbht}{Topological Baumgartner-Hajnal Theorem}
\newcommand{\irt}{Infinite Ramsey Theorem}
\newcommand{\schipperus}{Schipperus}

\newcommand{\ball}{\operatorname{Ball}}
\newcommand{\ballreal}{\ball_{\R}}
\newcommand{\homesubsets}[1]{${#1}$-element subsets}
\newcommand{\containsitslimit}{contains its limit point}

\newcommand{\monotoneset}[1]{\{{\monotoneelement{}}_{\al} \in {#1}:{\al} \in \N \cup \{\omega\}\}}
\newcommand{\monotonesetnolast}[1]{\{{\monotoneelement{}}_{\al}  \in {#1}:{\al} \in \N\}}

\newcommand{\sbinseq}{\mathcal{B}}
\newcommand{\fbinseq}{\mathcal{A}}
\newcommand{\fbin}{\theta^s}
\newcommand{\sbin}{\delta^s}
\newcommand{\cv}{c}

\newcommand{\map}[1]{\tilde{#1}}
\newcommand{\per}[4]{\mathcal P_{#1,#2}(\vec{#3},#4)}

\newcommand{\persimple}[5]{\mathcal P_{#1,#2}(\vec{#3},\vec{#4},\vec{#5})}
\newcommand{\persimplelong}[5]{\persing{#1}{#3_{1}}{#4_{1}}{#5_{1}}\times \dots \times \persing{#1}{#3_{#2}}{#4_{#2}}{#5_{#2}} }

\newcommand{\multP}[3]{P_{f,\vec{#1},\vec{#2},\vec{#3}}}

\newcommand{\squaref}[3]{\mathcal S_{#1,#2}(#3)}
\newcommand{\persing}[4]{\mathcal P_{#1}(#2,#3,#4)}
\newcommand{\squarefsing}[2]{\mathcal S_{#1}(#2)}

\newcommand{\gmatrix}[3]{\begin{pmatrix}
 #1_{1,1} & #1_{1,2} & \cdots & #1_{1,#3} \\
 #1_{2,1} & #1_{2,2} & \cdots & #1_{2,#3} \\
 \vdots   & \vdots   & \ddots & \vdots && \\
 #1_{#2,1}& #1_{#2,2}& \cdots & #1_{#2,#3}
\end{pmatrix}}
\newcommand{\gvector}[2]{\begin{pmatrix}
#1_{1} \\
#1_{2}\\
\vdots \\
#1_{#2}
\end{pmatrix}}
\newcommand{\gvectorstart}[3]{\begin{pmatrix}
#1_{#2}\\
\vdots \\
#1_{#3}
\end{pmatrix}}
\newcommand{\partiald}[2]{\frac{\partial {#1}}{\partial #2}}
\newcommand{\ordinaryd}[2]{\frac{\d {#1}}{\d #2}}
\newcommand{\polring}[3]{#1[#3_1,\dots,#3_{#2}]}
\newcommand{\polfield}[3]{#1(#3_1,\dots,#3_{#2})}
\newcommand{\btwn}[3]{{#1}\;,{#2} \leq {#1} \leq {#3}}

\newcommand{\av}{a}
\newcommand{\bv}{b}

\newcommand{\uv}{u}

\newcommand{\bbox}{\operatorname{Box}}

\newcommand{\rank}{\operatorname{Rank}}
\newcommand{\spanv}{\operatorname{Span}}
\newcommand{\image}{\operatorname{Image}}
\newcommand{\chr}{\operatorname{Char}}

\numberwithin{equation}{section}

\newcommand{\Z}{{\mathbb Z}} 
\newcommand{\Q}{{\mathbb Q}}
\newcommand{\R}{{\mathbb R}}
\newcommand{\C}{{\mathbb C}}
\newcommand{\N}{{\mathbb N}}
\newcommand{\FF}{{\mathbb F}}
\newcommand{\fe}{\overline{\mathbb F}}
\newcommand{\fq}{\mathbb{F}_q}
\newcommand{\feq}{\overline{\mathbb F}_q}

\newcommand{\rmk}[1]{\footnote{{\bf Comment:} #1}}

\renewcommand{\mod}{\;\operatorname{mod}}
\newcommand{\ord}{\operatorname{ord}}
\newcommand{\TT}{\mathbb{T}}
\renewcommand{\i}{{\mathrm{i}}}
\renewcommand{\d}{{\mathrm{d}}}
\newcommand{\HH}{\mathbb H}
\newcommand{\Vol}{\operatorname{vol}}
\newcommand{\area}{\operatorname{area}}
\newcommand{\tr}{\operatorname{tr}}
\newcommand{\norm}{\mathcal N} 
\newcommand{\intinf}{\int_{-\infty}^\infty}
\newcommand{\ave}[1]{\left\langle#1\right\rangle} 
\newcommand{\Var}{\operatorname{Var}}
\newcommand{\Prob}{\operatorname{Prob}}
\newcommand{\sym}{\operatorname{Sym}}
\newcommand{\disc}{\operatorname{disc}}
\newcommand{\CA}{{\mathcal C}_A}
\newcommand{\cond}{\operatorname{cond}} 
\newcommand{\lcm}{\operatorname{lcm}}
\newcommand{\Kl}{\operatorname{Kl}} 
\newcommand{\leg}[2]{\left( \frac{#1}{#2} \right)}  

\newcommand{\sumstar}{\sideset \and^{*} \to \sum}

\newcommand{\LL}{\mathcal L} 
\newcommand{\sumf}{\sum^\flat}
\newcommand{\Hgev}{\mathcal H_{2g+2,q}}
\newcommand{\USp}{\operatorname{USp}}
\newcommand{\conv}{*}
\newcommand{\dist} {\operatorname{dist}}
\newcommand{\CF}{c_0} 
\newcommand{\kerp}{\mathcal K}

\newcommand{\fs}{\mathfrak S}
\newcommand{\rest}{\operatorname{Res}} 
\newcommand{\af}{\mathbb A} 
\newcommand{\Ht}{\operatorname{Ht}}
\newcommand{\monic}{\mathcal M}

\title
[Infinite closed monochromatic subsets of a metric space] {Infinite closed monochromatic subsets of a metric space}
\author{Shai Rosenberg}

\email{shairos1@mail.tau.ac.il}

\maketitle

\begin{abstract}
Given a coloring of the \homesubsets{k} of an uncountable separable metric space, we show that there exists an infinite monochromatic subset which \containsitslimit{}.
\end{abstract}

\section{Introduction}
Given a coloring of the \homesubsets{k} of an infinite metric space $X$, the \irt{} guaranties the existence of an infinite monochromatic subset $\monoset \subseteq X$. However, if $\monosetelement \in X$ is a limit point of $\monoset$, the \irt{} does not imply that $\monosetelement \in \monoset$. Therefore, it may be that $\monoset$ does not contain any of its limit points. We show that if additionally, $X$ is assumed to be uncountable and separable, then such an infinite monochromatic set which \containsitslimit{} exists.

The \tbht{}, proved by \schipperus{} \cite{R. Schipperus}, provides a stronger result for the case $k=2$, $X=\R$. The latter states that if the pairs of real numbers are colored with $c$ colors, there is a monochromatic, well-ordered subset of arbitrarily large countable order type which is closed in the usual topology of $\R$. Applying the \tbht{} to the special case where the order type is $\omega + 1$, provides Theorem~\ref{thm main theorem} of this note for $k=2, X=\R$. For $k=3$ however, we show that the result in this note cannot be strengthened in some sense.

In Section~\ref{sec infinite closed monochromoatic subsets} we state and prove Theorem~\ref{thm main theorem}. The proof relies on the Axiom of Choice. In Section~\ref{sec examples} we show why the assumption that $X$ is uncountable is required, and why we cannot expect a stronger result to hold: that there exists a monochromatic subset of $X$, which contains more than one of its limit points.


\section{Infinite closed monochromatic subsets}\label{sec infinite closed monochromoatic subsets}

\begin{thm}\label{thm main theorem}
Let $X$ be an uncountable separable metric space with metric $d$. Let $k>0, c>0$, and let
$$\chi: [X]^k \rightarrow \{1,2,\dots,\cv\}$$
be a coloring of the \homesubsets{k} of $X$ with $\cv$ colors. Then there exists an infinite monochromatic subset $\monoset \subseteq X$, and there exists $\monosetelement \in \monoset$
s.t. $\forall \epsilon > 0, \exists \monosetelement_{\epsilon} \in \monoset$ s.t. $d(\monosetelement_{\epsilon},\monosetelement) < \epsilon$.
\end{thm}

We first list the notation which is used in the proof:
\begin{itemize}
\item For a set $A$, $[A]^k$ denotes the set of \homesubsets{k} of $A$.
\item $\fbin$ where $s \in \N, s\geq k-1$ denotes a function$$\fbin : \{k-1,k,k+1,\dots,s\} \rightarrow \N.$$ $\fbinseq$ is the set of all such functions, for all $s>0$.

\item $\sbin$ denotes a function which is defined on sets $\{i_1,i_2,\dots,i_{k-1}\}$ where $i_1,\dots,i_{k-1} \in \N$ and
    $i_1,\dots,i_{k-1} \leq s,$ and $\sbin(\{i_1,i_2,\dots,i_{k-1}\}) \in \{1,2,\dots,c\}$. Equivalently, $\sbin$ is a coloring of the \homesubsets{(k-1)} of $\{1,\dots,s\}$ with $c$ colors. $\sbinseq{}$ is the set of all such functions, for all $s>0$.

\item Let $\btwn{j}{1}{c}$, then
$$\neib{j}{a_1,\dots,a_{k-1}}:=\{a \in X : \chi(\{a,a_1\dots,a_{k-1}\})=j\}.$$

    \item For $x \in X, \varphi \in \R$, $B(x,\varphi)$ denotes the ball of radius $\varphi$ around $x$ with respect to $d$.
    \item $\{\denssetelement_n\}_{n \in \N}$: a countable subset of $X$ which is dense in $X$. Such set exists by our assumption that $X$ is separable.
        \item $\preccurlyeq$: well-ordering of $X$.

\end{itemize}

\subsection*{Proof overview}

\begin{enumerate}
\item \label{item monotoneset contains monochromatic set}Suppose there exists a subset of $X$: $\monotoneset{X}$
    , which satisfies the following properties:
\begin{enumerate}
\item Let $\al_1,\al_2,...,\al_{k-1} \in \N$ be distinct numbers, then $\forall \al \in \N, \alpha > \al_{k-1}$. \label{item monotone porperty}
 \begin{equation}\label{eq monotone property}
   \chi(\{\monotoneelement{}_{\al_1},\dots,\monotoneelement{}_{\al_{k-1}},\monotonelast{} \}) = \chi(\{\monotoneelement{}_{\al_1},\dots,\monotoneelement{}_{\al_{k-1}},\monotoneelement{}_{\al} \}).
 \end{equation}
 \item $\lim_{\al \to \infty} \monotoneelement{}_{\al} = \monotonelast$. \label{item limit property}

\end{enumerate}

Then by a known construction of Erd\"{o}s and Rado \cite{ER}, it follows from property \ref{item monotone porperty} that $\monotoneset{X}$ contains an infinite monochromatic subset $\monoset$ such that $\monotonelast \in \monoset$.

By setting $r = \monotonelast$ it follows from property \ref{item limit property} that $\monoset$ satisfies the requirements of the theorem. Therefore it remains to show that there exists such set $\monotoneset{X}$.

\item \label{item proof step contrary assumption}Suppose on the contrary, that such set $\monotoneset{X}$ does not exist.
\item \label{item proof step function construction}Define a function

\begin{equation}\label{eq function f definition}
    f: X \backslash \{x_1,x_2,\dots,x_{k-1}\} \rightarrow \fbinseq \times \sbinseq{},
\end{equation}
    where $x_{1},x_{2},\dots,x_{k-1}$ are the first minimal elements in $X$ according to $\preccurlyeq$. In other words, $f$ assigns two finite functions every $x \in X$, except for the first $k-1$ most minimal elements of $X$.
\item \label{item proof step show one to one}Show that the function $f$ is one-to-one.
\item The above is contradiction, since by the theorem assumption $X$ is uncountable, while $\fbinseq \times \sbinseq$ is countable.
\end{enumerate}

\subsection*{Proof of Theorem~\ref{thm main theorem}}

\begin{proof}
Suppose that there exists a subset $\monotoneset{X}$ which satisfies properties \ref{item monotone porperty} and \ref{item limit property} as stated in step~\ref{item monotoneset contains monochromatic set} of the proof overview. Then as described in step \ref{item monotoneset contains monochromatic set} of the proof overview, the theorem follows. Hence it remains to prove that such subset $\monotoneset{X}$ exists.

For the sake of completeness, we now describe briefly the construction of Erd\"{o}s and Rado which is mentioned in the proof overview. First, we define a coloring $\tilde{\chi}$, of the \homesubsets{(k-1)} of $\monotonesetnolast{X}$ by $\tilde{\chi}(\{u_{\al_1},\dots,u_{\al_{k-1}}\})=\chi(\{u_{\al_1},\dots,u_{\al_{k-1}},u_{\omega}\})$. Now applying the \irt{} to this coloring, we conclude that there exists an infinite subset $\tilde{\monoset}$ of $\monotonesetnolast{X}$, which is monochromatic with respect to $\tilde{\chi}$. Then $\tilde{\monoset} \cup \{u_{\omega}\}$ forms an infinite subset of $X$ which is monochromatic with respect to $\chi$. Hence taking $\monoset = \tilde{\monoset} \cup \{u_{\omega}\}$ we get $\monoset$ as required.

As stated step \ref{item proof step contrary assumption} of the proof overview, we now assume that such set does not exist.

We now proceed by defining a function $f$ as described in step \ref{item proof step function construction} of the proof overview:

Let $x \in X$. We now define $f(x)$. We first define $\fbin, \sbin$, where $\fbin \in \fbinseq$, $\sbin \in \sbinseq$ for some $s>0$, and $f(x)$ will be defined by $f(x)=(\fbin,\sbin)$.

For the purpose of defining $\fbin,\sbin$, we define recursively for each $i \leq s+1$, a set $U_i \subseteq X$ and $u_i \in U_i$. The general idea is that we want to encode some information in $\fbin, \sbin$ so that later we can reconstruct $U_i$ and $u_i$ without knowing $x$.

Let $u_{1},u_{2},\dots, u_{k-1} \in X$ be the first $k-1$ minimal elements in $X$ according to $\preccurlyeq$. Let $U_{k-1}=X$.



Now, suppose we already defined $U_{j}, u_{j}$ for $1\leq j \leq i$ and $$\sbin{}(\{j_1,j_2,\dots,j_{k-1}\}), \fbin{}(j)$$ for $k-1 \leq j \leq i-1$, $1 \leq j_1,\dots,j_{k-1} \leq i-1$. We now define $U_{i+1}, u_{i+1}$, $\fbin(i)$ and $$\sbin(\{i,j_1,\dots,j_{k-2}\}), \forall {j_1,\dots,j_{k-2} \leq i-1}.$$

Let $\denssetelement_n$ where $n$ is the minimal number such that the following holds:
\begin{equation}\label{eq x is in the ball}
x \in B(\denssetelement_n,2^{-i})
\end{equation}

Define
\begin{equation}\label{eq define fbin}
\fbin(i)=n.
\end{equation}
Define

\begin{equation}\label{eq second binary assignment}
  \sbin(\{i,j_1\dots,j_{k-2}\}): = \chi(\{x,u_{i}, u_{j_1},\dots,u_{j_{k-2}}\}),
\end{equation}

and define

\begin{equation}\label{eq definition of recorsive set}
U_{i+1} := \bigcap_{j_1,\dots, j_{k-2} < i}
\neib{\sbin(\{i,j_1\dots,j_{k-2}\})}{u_{i},u_{j_1}\dots,u_{j_{k-2}}}\cap B(\denssetelement_{\fbin(i)},2^{-i}) \cap U_{i}
\end{equation}

Define $u_{i+1}$ to be the minimal element in $U_{i+1}$ according to $\preccurlyeq$.

We have the following observations:
\begin{itemize}
\item $\forall i, x \in U_i$. In particular, $U_i$ is not empty.
\item After a finite number of steps $s>0$, we must get to a situation where $x$ is the minimal element in $U_{s+1}$. Otherwise, by setting $\monotonelast=x$, we get an infinite sequence $\monotoneset{X}$ which satisfies the properties as described in step \ref{item monotoneset contains monochromatic set} of the proof overview, which is a contradiction to our assumption.
\end{itemize}
 \item after $s$ steps, we define $f(x)=(\fbin, \sbin)$.

Step \ref{item proof step show one to one} in the proof overview: it remains to show that $f$ is one-to-one. Suppose we are given $f(x)$, or in the notation above $\fbin, \sbin$. In order to show that $f$ is one-to-one, we need to show that we can determine $x$.
In the process described above, the same $U_1,\dots,U_{s+1}$ and $u_1,\dots,u_{s+1}$ can be recovered without knowing $x$, since $\fbin,\sbin$ are known and by Eq.~\ref{eq definition of recorsive set} for each $i \leq s$ we can calculate $U_{i+1}$, and $u_{i+1}$ is the minimal element in $U_{i+1}$.

Therefore after reconstruction of the same $s$ steps we get $U_{s+1}$, and $x=u_{s+1}$ is the minimal element of $U_{s+1}$. In other words, if $x_1,x_2$ are such that $f(x_1)=f(x_2)$, then both $x_1$ and $x_2$ must be the minimal element of $U_{s+1}$. Since $U_{s+1}$ does not depend on $x_1,x_2$ we must have $x_1=x_2$. As stated in the proof overview this is a contradiction.

\end{proof}

\section{Counterexamples}\label{sec examples}

The following example shows a coloring for the case $k=3$, where any monochromatic subset has at most one limit point. Hence, for $k=3$ we cannot always expect a monochromatic set which contains more than one limit point of itself. In this sense, Theorem~\ref{thm main theorem} cannot be strengthened.

\begin{example}\label{ex 3 coloring of R}
Define a coloring of the \homesubsets{3} of $\R$ by
\begin{equation}\label{eq example only one limit point}
\chi(\{x_1,x_2,x_3\}) = \left\{
	\begin{array}{ll}
		0  & \mbox{if } \realdist{x_1}{x_2} \leq \realdist{x_2}{x_3} \\
		1 & \mbox{if } \realdist{x_1}{x_2} > \realdist{x_2}{x_3}
	\end{array},
\right.
\end{equation}
where $x_1 < x_2 < x_3$.

Suppose $\monoset \subseteq \R$ is an infinite subset of $\R$ which has two distinct limit points: $l_1,l_2 \in \R$, where $l_1 < l_2$ and $\realdist{l_1}{l_2}=h$. We now show that $\monoset$ cannot be monochromatic. Let $\monosetelement_1 \in \monoset$ s.t. $\realdist{\monosetelement_1}{l_l} < \frac{h}{5}$. Let $\monosetelement_2,\monosetelement_3 \in \monoset \mbox{ s.t. } \monosetelement_2 < \monosetelement_3$ and $\realdist{\monosetelement_2}{l_2}< \frac{h}{5}, \realdist{\monosetelement_3}{l_2} < \frac{h}{5}$.
Then $\realdist{\monosetelement_1}{\monosetelement_2} > h - \frac{h}{5}-\frac{h}{5} = \frac{3h}{5}$ while $\realdist{\monosetelement_2}{\monosetelement_3} \leq |\monosetelement_3-l_2|+|\monosetelement_2-l_2| \leq \frac{h}{5}+\frac{h}{5} = \frac{2h}{5}$. Hence $\realdist{\monosetelement_1}{\monosetelement_2} > \realdist{\monosetelement_2}{\monosetelement_3}$ and $\chi(\{\monosetelement_1,\monosetelement_2,\monosetelement_3\} )=1$. On the other hand, by a symmetric argument, taking $\monosetelement_1,\monosetelement_2 \in \R \mbox{ s.t. } \monosetelement_1 < \monosetelement_2, \realdist{\monosetelement_1}{l_1}< \frac{h}{5}, \realdist{\monosetelement_2}{l_1} < \frac{h}{5}$ and $\monosetelement_3 \in \R$ s.t. $\realdist{\monosetelement_3}{l_2} < \frac{h}{5}$, we get $\realdist{\monosetelement_1}{\monosetelement_2} < \realdist{\monosetelement_2}{\monosetelement_3}$. Hence $\chi(\{\monosetelement_1,\monosetelement_2,\monosetelement_3\}=0$. This shows that $\monoset$ is not monochromatic.
\end{example}

In the following two examples we show why the assumption that $X$ is uncountable is required in Theorem~\ref{thm main theorem}. This is done by defining a coloring of the \homesubsets{k} of $\Q$, for which an infinite monochromatic subset does not contain any of its limit points. In the following two examples we assume $(\denssetelement_n)_{n \in \N}$ is an enumeration of $\Q$, and $\preccurlyeq$ is an order on $\Q$, defined by $\denssetelement_{i} \preccurlyeq \denssetelement_{j}$ if and only if $i \leq j$.

Example~\ref{ex 2 coloring of Q} is the construction of Sierpinsky \cite{Sierpnisky}, applied to $\Q$.
\begin{example}\label{ex 2 coloring of Q}
Define a coloring of pairs of numbers in $\Q$ by
\begin{equation}\label{eq example 2 coloring of Q}
\chi(\{x_1,x_2\}) = \left\{
	\begin{array}{ll}
		0  & \mbox{if } x_2 \preccurlyeq x_1  \\
		1 & \mbox{if } x_1 \preccurlyeq x_2
	\end{array},
\right.
\end{equation}
where $x_1 < x_2$.
\end{example}

\begin{example}\label{ex 3 coloring of Q}
Define a coloring of \homesubsets{3} of $\Q$ by
\begin{equation}\label{eq example only one limit point}
\chi(\{x_1,x_2,x_3\}) = \left\{
	\begin{array}{ll}
		0  & \mbox{if } x_1 \preccurlyeq x_2 \mbox{ and }x_3 \preccurlyeq x_2\\
		1 & otherwise
	\end{array},
\right.
\end{equation}
where $x_1 < x_2 < x_3$.
\end{example}

We claim the following.
\begin{claim}\label{claim examples over Q cannot contain closed subset}
\begin{enumerate}

\item \label{claim item coloring 2 of Q}In the coloring of Example~\ref{ex 2 coloring of Q} there exist no infinite monochromatic subset which \containsitslimit{}.
    \item\label{claim item coloring 3 of Q} In the coloring of Example~\ref{ex 3 coloring of Q} there exists no infinite subset, all whose \homesubsets{3} are colored $1$, which \containsitslimit{}, and there exists no subset, all whose \homesubsets{3} are colored $0$ which contains more that $3$ elements.
        \end{enumerate}
\end{claim}

Part~\ref{claim item coloring 2 of Q} of Claim~\ref{claim examples over Q cannot contain closed subset} shows that for coloring \homesubsets{2} the requirement in Theorem~\ref{thm main theorem} that $X$ is uncountable, is necessary.

We may ask whether for some $l \geq k$ it holds that when coloring the \homesubsets{k} of $\Q$ there must be either a subset of size $l$, all whose \homesubsets{k} are colored $0$, or an infinite subset which \containsitslimit{}, all whose \homesubsets{k} are colored $1$. For $l=k$ this holds trivially. Part~\ref{claim item coloring 3 of Q} of Claim~\ref{claim examples over Q cannot contain closed subset} shows that for $k=3$, this does not generally hold for $l>3$.

We first prove the following claim, which we use in the proof of Claim~\ref{claim examples over Q cannot contain closed subset}.

\begin{claim}\label{claim special triple}
If $\monoset$ is an infinite subset of $\Q$ which \containsitslimit{}, then there exist $\monosetelement_1,\monosetelement_2,\monosetelement_3 \in \monoset$ s.t. $\monosetelement_1 < \monosetelement_2 < \monosetelement_3$, $\monosetelement_1 \preccurlyeq \monosetelement_2$ and $\monosetelement_3 \preccurlyeq \monosetelement_2$.
\end{claim}

\begin{proof}
Let $\monosetelement_3$ be a limit point of $\monoset$. For $\varphi \in \R$, Let $\monoset_{\varphi}' = \{b \in \monoset: \varphi < b < \monosetelement_3\}$ and let $\monoset_{\varphi}'' = \{b \in \monoset: \varphi > b > \monosetelement_3\}$. Since $\monosetelement_3$ is a limit point of $\monoset$, it must be that either $\monoset_{\varphi}'$ is not empty for all $\varphi < r_3$, or $\monoset_{\varphi}''$ is not empty for all $\varphi > \monosetelement_3$.

Assume first that $\monoset{}_{\varphi}'$ is not empty for all $\varphi < \monosetelement_3$ . There exist only finitely many $x \in \monoset'$ s.t. $x \preccurlyeq r_3$. Hence for some $\varphi_1 \in \R$, $0 < \varphi_1 < \monosetelement_3$, $\forall x \in \monoset{}_{\varphi_1}'$, we have $\monosetelement_3 \preccurlyeq x$. Let $\monosetelement_1 \in \monoset{}_{\varphi_1}'$. There exists $\varphi_2$ s.t. $\monosetelement_1 < \varphi_2 < \monosetelement_3$ and $\forall x \in \monoset_{\varphi_2}'$, we have $\monosetelement_1 \preccurlyeq x$. Let $\monosetelement_2 \in \monoset{}_{\varphi_2}'$. Then $\monosetelement_1 < \monosetelement_2 < \monosetelement_3$ and $\monosetelement_3 \preccurlyeq \monosetelement_1 \preccurlyeq \monosetelement_2$ as
desired.

Now, if $\monoset_{\varphi}''$ is not empty for all $\varphi>r_3$, by a symmetric argument it follows that
$\monosetelement_3 < \monosetelement_2 < \monosetelement_1$ and $\monosetelement_3 \preccurlyeq \monosetelement_1 \preccurlyeq \monosetelement_2$.
Swapping $\monosetelement_1$ and $\monosetelement_3$ gives us that $\monosetelement_1,\monosetelement_2,\monosetelement_3$ satisfy the requirements as desired.
\end{proof}

We now show why Claim~\ref{claim examples over Q cannot contain closed subset} holds.
\begin{proof}
\begin{enumerate}

\item Let $\chi$ be a coloring as in Example~\ref{ex 2 coloring of Q}. Let $\monoset$ be a subset of $\Q$ which \containsitslimit{}. Let $\monosetelement_1,\monosetelement_2,\monosetelement_3 \in \monoset$ as in Claim~\ref{claim special triple}. Then $\chi(\{\monosetelement_1,\monosetelement_2\}) \not= \chi(\{\monosetelement_2,\monosetelement_3\}$. Hence the claim follows.
\item Let $\chi$ be a coloring as in Example~\ref{ex 3 coloring of Q}, and let $\monoset$ be an infinite subset of $\Q$ which \containsitslimit{}. Let $r_1,r_2,r_3 \in \monoset$ as in Claim~\ref{claim special triple}. Then $\chi(\{\monosetelement_1,\monosetelement_2,\monosetelement_3\}) = 0$. Hence not all $3$-element subsets of $\monoset$ are colored $1$. On the other hand, suppose on the contrary that there exist $r_1,r_2,r_3,r_4 \in \Q$ s.t. $\monosetelement_1<\monosetelement_2<\monosetelement_3<\monosetelement_4$ and all $3$-element subsets of $\{\monosetelement_1,\monosetelement_2,\monosetelement_3,\monosetelement_4\}$ are colored $0$. Then considering $\monosetelement_1,\monosetelement_2,\monosetelement_3$ we must have $\monosetelement_3 \preccurlyeq \monosetelement_2$. On the other hand, considering $\monosetelement_2,\monosetelement_3,\monosetelement_4$ we have $\monosetelement_2 \preccurlyeq \monosetelement_3$. The last is a contradiction. Hence such monochromatic subset of size $4$ does not exist.
\end{enumerate}
\end{proof}

\end{document}